\documentclass[12pt]{article}
\usepackage{amsmath,amssymb,amsthm}
\usepackage{tikz}
\newtheorem{lemma}{Lemma}
\newtheorem{thm}[lemma]{Theorem}
\usepackage{caption}
\usepackage{subcaption}

\title{Equidissection of darts}
\author{Yusong Deng, Iwan Praton}
\date{}
\begin{document}
\maketitle
\abstract{We define the dart $D(a)$ to be the nonconvex quadrilateral whose vertices are $(0,1), (1,1), (1,0), (a,a)$ (in counterclockwise order), with $a>1$. Such a dart can be dissected into any even number of equal-area triangles. Here we investigate darts that can be dissected into an odd number of equal-area triangle.}

\section{Introduction}

An $m$-equidissection of a planar polygon $K$ is a dissection of $K$ into $m$ triangles of equal area. The spectrum $S(K)$ of $K$ is the set of integers $m$  for which an $m$-equidissection exists. In a pioneering paper, Monsky \cite{M} showed that if $K$ is a square, then $S(K)$ consists of all even positive integers and no odd ones. Such a striking result (and its even more striking proof) inspired others to investigate the spectrum of other polygons, e.g., \cite{KS}, \cite{J}, \cite{JM}. An interesting and thorough discussion of this topic can be found in chapter 5 of the book \cite{SS}.

Jepsen, Sedberry, and Hoyer \cite{JSH}  looked at the kite-shaped quadrilateral $Q(a)$ whose vertices are $(0, 0), (1, 0), (0, 1), (a, a)$ where $a > 1/2$. They observed that $S(Q(a))$ contains 2, and since a triangle can be dissected into any number of equal area triangles, $S(Q(a))$ contains all positive even integers. This leads to the main question of their paper: which odd integers, if any, are in $S(Q(a))$? They provide answers for certain values of $a$. 

Inspired by Penrose tilings, we note that the kites have companion polygons which we call darts. The dart $D(a)$ is a quadrilateral whose vertices are $(0,1), (1,1), (1,0), (a,a)$ (in counterclockwise order), with $a>1$.  Of course  $D(a)$ is not convex, but it still makes sense to define the spectrum of nonconvex polygons. In particular, $S(D(a))$ contains all positive even integers for the same reason that $S(Q(a))$ does, so we ask the same question as before: which odd integers are in $S(D(a))$?

\section{Results}

Our main tool is the same one used in \cite{JSH} (and earlier in \cite{KS} and \cite{SS}). We state it explicitly as a lemma.

\begin{lemma}
Suppose the polygon $K$ is dissected into triangles with areas $A_1,\ldots, A_m$, and these areas are in proportion
\[
A_1 : A_2:\cdots : A_m = t_1 : t_2 : \cdots : t_m,
\]
where $t_i$ is an integer ($1\leq i\leq m$). Then $K$ has a $t$-equidissection, where $t=t_1+\cdots + t_m$. 
\end{lemma}
\begin{proof}
We dissect the $i$th triangle into $t_i$ equal-area triangles, thus obtaining the required $t$-equidissection of $K$.
\end{proof}

The first result in \cite{JSH} states that $S(Q(r/2s))$ contains all integers of the form $r+2k$, where $k\geq 0$. We show that a similar result holds for darts $D(r/2s)$.

\begin{thm}
Let $a=r/(2s)$, where $r$ and $s$ are relatively prime positive integers with $r$ odd and $r>2s$. Then the integers $r+2k$ $(k\geq 0)$ are in $S(D(a))$. 
\end{thm}
\begin{proof}
Dissect $D(a)$ into three triangles as shown in figure 1, left. This partition is adapted from \cite{JSH}.

\begin{figure}[h]
\centering
\begin{subfigure}[b]{0.5\textwidth}
\centering
\begin{tikzpicture}[scale=1.8]
\draw[thick] (1,1)--(1,0)--(7/4,7/4)--(0,1)--(1,1);
\draw [red] (0,1)--(1.58,1.35)--(1,1);
\node at (2,1.3){$(p,q)$}; \node at (2.1,7/4){$(a,a)$};
\node at (1,1.12){ 1}; \node at (1.2,0.8){2};
\node at (1.32,1.42){3};
\end{tikzpicture}
\end{subfigure}
\begin{subfigure}[b]{0.4\textwidth}
\centering
\begin{tikzpicture}[scale=1.8]
\draw[thick] (1,1)--(1,0)--(7/4,7/4)--(0,1)--(1,1);
\draw [red](1.2,0.44)--(0.87,1.36);
\node at (1.5,0.4){$(p,q)$};
\node at (1.1,0.5){1}; \node at (0.6,1.15){2};
\node at (1.3,1.3){3};
\end{tikzpicture}
\end{subfigure}
\caption{Two partitions of $D(a)$}
\end{figure}
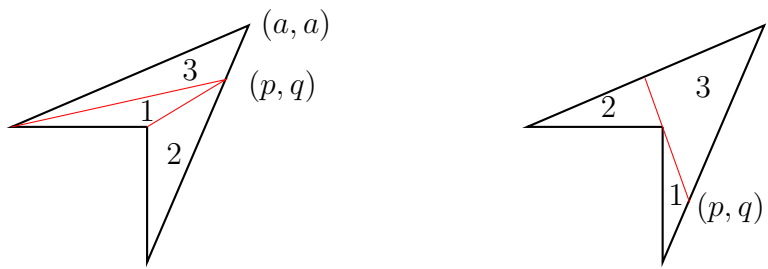

We choose $q=r (t-r-2s)/2st$ and $p=q+(r-2s)/t$ where $t$ is an odd integer and $t\geq r$. It is straightforward to check that (i) $(p,q)$ lies on the indicated line; (ii) $q\geq 1$; (iii) $a>p$. Thus it is possible to partition $D(a)$ into 3 triangles as shown in the figure; the non-convexity is not an issue. It is also straightforward to calculate that area 1 is $(r-2s)(t-r)/(4st)$, area 2 is $(r-2s)(t-r+2s)/(4st)$, and area 3 is $(2r-2s)(r-s)/(4st)$. Therefore these areas are in the proportion $(t-r) : (t-r+2s) : 2(r-s)$. Since $r$ and $t$ are both odd, we can simplify the proportions to $(t-r)/2 : (t-r+2s)/2 : (r-s)$. These sum up to $t$, so by Lemma 1, $D(a)$ has a $t$-equidissection.
\end{proof}

A natural question arises: in the situation of Theorem 1, does $D(a)$ contain an odd number smaller than $r$? For $Q(a)$, \cite{JSH} showed that the answer is yes when $r$ is not a prime number. It doesn't seem straightforward to adapt their method to our situation. But we can take advantage of the nonconvexity of $D(a)$ to get a similar result.

\begin{thm}
Let $a$ be as in the previous theorem, and suppose $s$ is even. Then $S(D(a))$ contains odd numbers less than $r$.
\end{thm}
\begin{proof}
We partition $D(a)$ into 3 triangles as shown in figure 1, right; for this theorem we choose the central line so that its equation is $x+y=2$. This implies that $p=(3r-4s)/(2r-2s)$, $q=r/(2r-2s)$. Area 1 and area 2 are each equal to $(r-2s)/(4r-4s)$, and area 3 is $r/(2s)-1-(r-2s)/2r-2s)=(r-2s)(2r-4s)/((4r-4s)s)$. Thus the three areas are in the proportion 
\[
s : s : 2r-4s, \text{ i.e., } \frac{s}2 : \frac{s}2 : r-2s
\]
since $s$ is even. These sum up to $r-s$; by Lemma 1, $S(D(a))$ contains the odd integer $r-s<r$. 
\end{proof}

It's natural to wonder if the same result holds for odd values of $s$. It does not hold for all  $D(r/(2s))$ with $s$ odd, since $D(3/2)$ cannot be $t$-equidissected if $t$ is odd and smaller than 3. But it does hold for some $D(r/(2s))$.

\begin{thm}
Let $a=r/(2s)$ be as in theorem 1, and suppose $s$ is odd and $r-s$ has an odd factor larger than $r-2s$. Then $S(D(a))$ contains odd integers less than $r$.
\end{thm}
\begin{proof}
Let $r-s=ct$ where $t$ is odd and $t>r-2s$. We partition $D(a)$ into three triangles as shown in figure 1, right. This time we choose $q=r(t-r+2s)/(2st)$. Then $p=q+(r-2s)/t$, $p'=r(r-t)/(2s(2r-2s-t))$ and $q'=p'+(r-2s)/(2r-2s-t)$ (after some calculation). Therefore
\begin{itemize}
\item
area 1 is $\frac12 (p-1) = \frac12 \ \frac{r-2s}{2s}\ \frac{t-r+2s}{t}$;
\item
 area 2 is $\frac12 (q'-1)=\frac12 \ \frac{r-2s}{2s} \ \frac{r-t}{2r-2s-t}$;
 \item
area 3 is $a-1-\text{area 1} - \text{area 2} = \frac12 \ \frac{r-2s}{2s}\ \frac{(2r-2s)(r-2s)}{t(2r-2s-t)}$.
\end{itemize}
We conclude that
\[
\text{area 1} : \text{area 2} : \text{area 3} = (t-r+2s)((2r-2s-t) : t(r-t) : (2r-2s)(r-2s)
\]
which simplifies to
\begin{align*}
t(3r-4s-t)&-(2r-2s)(r-2s) : t(r-t) : (2r-2s)(r-2s) \\
&= (3r-4s-t)-2c(r-2s) : (r-t) : 2c(r-2s).
\end{align*}
Now $r$ and $t$ are odd, so the ratios are as the integers $(3r-4s-t)/2$, $(r-t)/2$, and $c(r-2s)$. These add up to $2(r-s)-t$. By Lemma 1, $2(r-s)-t$ is in $S(D(a))$. Since $t>r-2s$, we have $r-s-t<s$, so  $2(r-s)-t= r-s + r-s-t<r$. Thus $S(D(a))$ contains an odd integer smaller than $r$. 
\end{proof}

For example, if $a=11/10$, then $r-s= 6$ and  $t=3$. So $S(D(a))$ contains $2(r-s)-t=9$ as well as $11$.

A bit more generally, if $r$ and $s$ are odd positive integers, with $2s<r<3s$ and $r\not\equiv s \bmod 4$, then $S(D(r/(2s)))$ contains odd integers smaller than $r$. 

We now turn our attention to the case where $a$ is a quadratic irrational. Since our results here turn out to be also applicable to kites, we first use an affine transformation to change our darts into ``non-convex kites''. We slightly generalize the definition of  $Q(a)$ to be the quadrilateral with vertices $(0,0)$, $(1,0)$, $(0,1)$, and $(a,a)$ where $a$ does not have to be larger than $1/2$; recall that a kite is  $Q(a)$ with $a>1/2$. The transformation
\[
(x,y)\mapsto \frac1{2a-1} \big((a-1)x+a(1-y), (a-1)y+a(1-x)\big)
\]
takes the dart $D(a)$ into $Q(a')$, where $a'=(a-1)/(2a-1)<1/2$. So $Q(a')$ is not a kite; it is more dart-like. Any equidissection of $Q(a')$ has an equivalent equidissection of $D(a)$.

From now on our darts are of the form $Q(a)$ with $a<1/2$. Jepsen, Sedberry, and Hoyer \cite{JSH} showed that there are infinitely many radicals $a>1/2$ (such as $\sqrt{3}/2$, $\sqrt{5}/4$, $\sqrt{21}/4$, etc) so that the kite $Q(a)$ has an odd equidissection. We show a similar result for darts. In fact, we show this for all odd radicals.

\begin{thm}
For $k\geq 1$, let $a=\sqrt{2k+1}/(4k+2)$. Then $S(Q(a))$ contains an odd integer.
\end{thm}
\begin{proof}
Partition $Q(a)$ into five triangles as shown in figure 2, left. It is similar to the zig-zag partitions in \cite{KS} (figure 1) and in \cite{JSH} (figure 2).

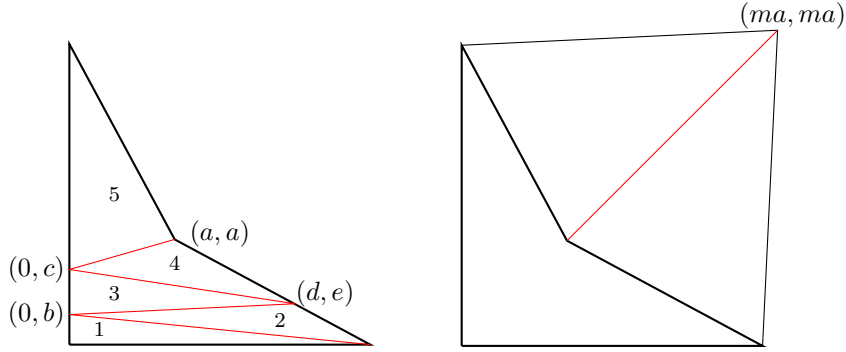
\begin{figure}[h]
\centering
\begin{subfigure}[b]{0.5\textwidth}
\centering
\begin{tikzpicture}[scale=4]
\draw [thick] (0,0)--(1,0)--(0.35,0.35)--(0,1)--(0,0);
\draw[red] (1,0)--(0,0.1)--(0.75,0.136)--(0,0.25)--(0.35,0.35);
\node at (-0.11,0.1){\footnotesize $(0,b)$};
\node at (-0.11,0.25){\footnotesize $(0,c)$};
\node at (0.85,0.17){\footnotesize $(d,e)$};
\node at (0.5,0.37){\footnotesize $(a,a)$};
\node at (0.1,0.05){\scriptsize $1$};
\node at (0.7,0.08){\scriptsize $2$};
\node at (0.15,0.17){\scriptsize $3$};
\node at (0.35,0.27){\scriptsize $4$};
\node at (0.15,0.5){\scriptsize $5$};
\end{tikzpicture}
\end{subfigure}
\begin{subfigure}[b]{0.4\textwidth}
\centering
\begin{tikzpicture}[scale=4]
\draw [thick] (0,0)--(1,0)--(0.35,0.35)--(0,1)--(0,0);
\draw (1,0)--(1.05,1.05)--(0,1);
\node at (1.1,1.1){\footnotesize $(ma,ma)$};
\draw[red] (0.35,0.35) -- (1.05,1.05);
\end{tikzpicture}
\end{subfigure}
\caption{Partitions of a dart and a kite}
\end{figure}

Set $m=2k+1$ and choose $b=\sqrt{m}/m^2$, $c=1/m$, and $d=m(1+\sqrt{m})/((m-1)(2m-1))$. We check that $c>b$, $a<d<1$, and $c<a/(1-a)$, implying that the dart can be partitioned as shown in the figure. Then
\begin{itemize}
\item
area 1 is  $\frac{b}{2} = \frac{a}{m}$
\item
area 2 is $\frac{(1-d)}{2(1-a)} (a-b+ab) = \frac{a(m^2-3m+1)}{m(2m-1)}$, 
\item
area 3 is $\frac{(c-b)d}{2}=\frac{a}{2m-1}$,
\item area 4 is $\frac{(d-a)}{2(1-a)}(a-c+ac)=\frac{ak}{m(2m-1)}$,
\item
area 5 is $\frac{a(1-c)}{2}=\frac{ak}{m}$
\end{itemize}

Thus the ratios of these areas are as
\[
2m-1 : (m^2-3m+1 : m : k : k(2m-1),
\]
whose sum is $8k^2+6k+1$, an odd integer. By Lemma 1, $Q(a)$ has an odd equidissection.
\end{proof}

It turns out that this result also applies to kites.

\begin{thm}
Suppose $m>1$ is an odd integer. Then the spectrum of the kite $Q(\sqrt{m}/2)$ contains an odd integer. 
\end{thm}
For $m=3$, this is theorem 3 in \cite{JSH}.
\begin{proof}
Let $a=\sqrt{m}/(2m)$ as in theorem 5. Then $Q(\sqrt{m}/2)=Q(ma)$. We partition $Q(ma)$ as shown in figure 2, right. By theorem 5, the dart $Q(a)$ has a $u$-equidissection, where $u$ is odd.  Each large triangle in the figure has area $(m-1)a/2= u(m-1)/2 \cdot a/u$, and so can be dissected into $u(m-1)/2$ triangles, each with area $a/u$. Thus the whole kite can be dissected into $u+u(m-1)=um$ equal-area triangles, and of course $um$ is odd. 
\end{proof}

\section*{Acknowledgement}
The results in this article are part of an undergraduate honors thesis by Yusong Deng, under the direction of Iwan Praton.

\end{document}